\documentclass[12pt]{amsart}
\usepackage{amsfonts,amssymb,amscd,amsmath,enumerate,verbatim,color}
\usepackage[latin1]{inputenc}
\usepackage{amscd}
\usepackage{latexsym}
\usepackage{graphicx} 
\usepackage{tikz}
\usetikzlibrary{intersections, calc, arrows}

\usepackage{mathptmx}

\def\NZQ{\mathbb}               

\def\ZZ{{\NZQ Z}}
\def\RR{{\NZQ R}}

\def\ab{{\mathbf a}}

\def\eb{{\mathbf e}}

\def\xb{{\mathbf x}}

\def\opn#1#2{\def#1{\operatorname{#2}}} 
\opn\gr{gr}

\def\Pc{{\mathcal P}}

\newtheorem{Theorem}{Theorem}[section]
\newtheorem{Lemma}[Theorem]{Lemma}

\newtheorem{Proposition}[Theorem]{Proposition}

\theoremstyle{definition}

\let\epsilon\varepsilon
\let\phi=\varphi
\let\kappa=\varkappa
\opn\dis{dis}
\opn\height{height}
\opn\dist{dist}
\def\pnt{{\raise0.5mm\hbox{\large\bf.}}}

\opn\Lex{Lex}
\opn\conv{conv}
\opn\codeg{codeg}
\opn\codim{codim}
\opn\int{int}

\opn\reg{reg} 
\begin{document}

\title{Edge rings with $q$-linear resolutions}
\author{Kenta Mori, Hidefumi Ohsugi and Akiyoshi Tsuchiya}

\address{Kenta Mori,
	Department of Mathematical Sciences,
	School of Science and Technology,
	Kwansei Gakuin University,
	Sanda, Hyogo 669-1337, Japan} 
\email{k-mori@kwansei.ac.jp}

\address{Hidefumi Ohsugi,
	Department of Mathematical Sciences,
	School of Science and Technology,
	Kwansei Gakuin University,
	Sanda, Hyogo 669-1337, Japan} 
\email{ohsugi@kwansei.ac.jp}

\address{Akiyoshi Tsuchiya,
Graduate School of Mathematical Sciences,
University of Tokyo,
Komaba, Meguro-ku, Tokyo 153-8914, Japan} 
\email{akiyoshi@ms.u-tokyo.ac.jp}


\begin{abstract}
In the present paper, we give a complete classification of connected simple graphs whose edge rings have a $q$-linear resolution with $q \geq 2$. In particular, we show that the edge ring of a finite connected simple graph with a $q$-linear resolution, where $q \geq 3$, is a hypersurface, which was conjectured by Hibi, Matsuda, and Tsuchiya.
\end{abstract}

\maketitle

\section{Introduction}
Let $K$ be a field and $K[\xb]:=K[x_1,\ldots,x_n]$ the polynomial ring with $n$ variables over $K$. 
For a finite simple graph $G$ on the vertex set $[n]:=\{1,\ldots,n\}$, the \textit{edge ring} $K[G]$ of $G$ is the $K$-subalgebra of $K[\xb]$ generated by the quadratic monomials $x_i x_j$ corresponding to the edges $\{i,j\}$ of $G$. 
Recently, edge rings and the associated lattice polytopes, which are called \textit{edge polytopes}, have been studied from 
the viewpoints of combinatorics, graph theory, geometric algebra, and commutative algebra. 
In particular, there has been significant interest in understanding the minimal free resolutions of $K[G]$ for several classes of graphs \cite{BOV,HKO,HMT,GM,NN,OHquad,OHkoszul,Tbipartite}.
We are interested in which edge ring has a $q$-linear resolution, where $q \geq 2$.
Previously, in \cite{OHquad}, Hibi and the second author gave an algebraic characterization of finite connected simple graphs whose edge rings have $2$-linear resolutions. 

\begin{Proposition}[{\cite[Theorem 4.6]{OHquad}}]
\label{thm:OH}
Let $G$ be a finite connected simple graph on $[n]$. Then $K[G]$ has a $2$-linear resolution if and only if $K[G]$ is isomorphic to the polynomial ring in $n-2\delta$ variables over the edge ring $K[K_{2,\delta}]$ of the complete bipartite graph $K_{2,\delta}$.
\end{Proposition}

Furthermore, in \cite{HMT}, Hibi, Matsuda, and the third author showed an algebraic property of finite connected simple graphs whose edge rings have $3$-linear resolutions.

\begin{Proposition}[{\cite[Theorem 0.1]{HMT}}]
\label{thm:HMT}
	The edge ring of a finite connected simple graph with a $3$-linear resolution is a hypersurface.
\end{Proposition}

In addition, they conjectured that the edge ring of a finite connected simple graph with a $q$-linear resolution, where $q \geq 4$, is a hypersurface (\cite[Conjecture 0.2]{HMT}).
Recently, this conjecture has been proved for the cases of chordal graphs \cite{NN} and bipartite graphs \cite{Tbipartite}.

In the present paper, we prove this conjecture for any finite connected simple graph.

\begin{Theorem}
\label{thm:main}
	The edge ring of a finite connected simple graph with a $q$-linear resolution, where $q \geq 4$, is a hypersurface.
\end{Theorem}

Moreover, using Propositions \ref{thm:OH} and \ref{thm:HMT} and Theorem \ref{thm:main}, we give a complete classification of connected simple graphs whose edge rings have a $q$-linear resolution with $q \geq 2$.

\begin{Theorem}
\label{2lr}
Let $G$ be a finite connected simple graph.
Then $K[G]$ has a $2$-linear resolution if and only if 
$G$ satisfies one of the following{\rm :}
\begin{itemize}
\item[{\rm (i)}]
$G$ is obtained by adding some trees to a complete bipartite graph $K_{2,\delta}$.

\item[{\rm (ii)}]
$G$ is a non-bipartite graph and
obtained by adding one edge to a graph
satisfying condition {\rm (i)} above.
\end{itemize}
\end{Theorem}

A vertex $v$ of a connected graph $G$ is called a {\em cut vertex}
if the graph obtained by the removal of $v$ from $G$ is disconnected.
Given a graph $G$, a {\em block} of $G$ is a maximal connected subgraph of $G$ with no cut vertices.

\begin{Theorem}
\label{qlr}
Let $q  \geq 3$, and let $G$ be a finite connected simple graph
with non-edge blocks $B_1,\ldots,B_s$.
Then $K[G]$ has a $q$-linear resolution if and only if 
$G$ satisfies one of the following{\rm :}
\begin{itemize}
\item[{\rm (i)}]
$s=1$ and $B_1$ is an even cycle of length $2q$.

\item[{\rm (ii)}]
$s=1$ and $B_1$ is a non-bipartite graph
obtained by adding a path to an even cycle of length $2q$.

\item[{\rm (iii)}]
$s=2$ and $B_1$ is an even cycle of length $2q$,
and $B_2$ is an odd cycle, or vice versa.

\item[{\rm (iv)}]
$s=2$ and $B_1$ and $B_2$ are odd cycles
having one common vertex.
In addition, $r_1 + r_2 = 2q$, where $r_i$
is the length of $B_i$.

\item[{\rm (v)}]
$s=2$ and $B_1$ and $B_2$ are odd cycles
without a common vertex.
In addition, the length of the shortest path from a vertex of $B_1$
to a vertex of $B_2$ is $q- ( r_1 + r_2)/2$, where $r_i$
is the length of $B_i$.
\end{itemize}
\end{Theorem}

The present paper is organized as follows.
In Section~\ref{section:ql}, we introduce the notion of $q$-linear resolutions,
and give necessary conditions for the edge rings to have $q$-linear resolutions.
In particular, the degree of the edge polytopes is important (Lemma~\ref{deglemma}).
In Section~\ref{sec:edgepolyandring}, in order to give several 
necessary conditions for the edge rings to have $q$-linear resolutions,
we study the degree of the edge polytopes and 
the minimal set of generators of toric ideals of edge rings.
Finally, in Section~\ref{sec:last},  we give proofs for Theorems~\ref{thm:main}, \ref{2lr}, and \ref{qlr}.

\section{Toric rings with $q$-linear resolutions}
\label{section:ql}

Let $S = K[x_{1}, \ldots, x_{n}]$ denote the polynomial ring in $n$ variables over a field $K$ with each $\deg x_{i} = 1$.  Let $0 \neq I \subset S$ be a homogeneous ideal of $S$ and 
\[
0 \to \bigoplus_{j \geq 1} S(-j)^{\beta_{h, j}} \to \cdots \to \bigoplus_{j \geq 1} S(-j)^{\beta_{1, j}} \to S \to S/I \to 0
\]
a (unique) graded minimal free $S$-resolution of $S/I$.  The {\em Castelnuovo-Mumford regularity} of $S/I$ is 
\[
\reg (S/I) = \max\{ j - i : \beta_{i, j} \neq 0 \}. 
\] 
We say that $S/I$ has a {\em $q$-linear resolution} if $\beta_{i, j} = 0$ for each $1 \leq i \leq h$ and for each $j \neq q + i - 1$.  If $S/I$ has a $q$-linear resolution, then
\begin{itemize}
\item
 $\reg(S/I) = q - 1$ and 
\item
$I$ is generated by homogeneous polynomials of degree $q$.  
\end{itemize}
See, e.g., \cite{BH} and \cite{binomialideals} for detailed information about regularity and linear resolutions.  
In addition, the following necessary condition is known.

\begin{Lemma}[{\cite[Proposition 1.9 (d)]{EG}}]
\label{EGlemma}
If $S/I$ has a $q$-linear resolution, then 
the number of generators of $I$ is at least $\binom{c+q-1}{c-1}$, where 
$c$ is the codimension of $S/I$.
\end{Lemma}

A {\em lattice polytope} $\Pc \subset \RR^n$ is a polytope such that
any vertex of $\Pc$ belongs to $\ZZ^n$.
Let $K[\xb^{\pm 1}, s] =K[x_{1}^{\pm 1}, \ldots, x_{n}^{\pm 1}, s]$ be a Laurent polynomial ring in $n+1$ variables
over a field $K$.
If $\Pc \cap \ZZ^n = \{\ab_1,\ldots, \ab_m\}$,
then the {\em toric ring} $K[\Pc]$ of $\Pc$ is
the $K$-subalgebra of $K[\xb^{\pm 1}, s]$ generated by the monomials
${\bf x}^{\ab_1} s,\ldots,{\bf x}^{\ab_m} s \in K[\xb^{\pm 1}, s]$.
Furthermore, the {\em toric ideal} $I_{\Pc}$ of $\Pc$ is 
the defining ideal of $K[\Pc]$, i.e., the kernel of 
a surjective ring homomorphism
$\pi : K[y_1,\ldots, y_m] \rightarrow K[\Pc]$ defined by
$\pi(y_i) = {\bf x}^{\ab_i} s$ for $i=1,2,\ldots,m$.
It is known that $I_\Pc$ is generated by homogeneous binomials.
See \cite[Chapter 3]{binomialideals} for an introduction to toric rings and ideals.

Let $\Pc$ and $\Pc'$ be lattice polytopes in $\RR^n$.
The toric ring $K[\Pc']$ is called a {\em combinatorial pure subring}
 of the toric ring $K[\Pc]$ if $\Pc'$ is a face of $\Pc$.
Although this definition is different from that in \cite{cps},
they are equivalent (see \cite{Ocps}).
It is known \cite[Corollary~2.5]{cps} that,
if $K[\Pc']$ is a combinatorial pure subring of $K[\Pc]$, then
\[
\beta_{i,j} (I_{\Pc'}) \le 
\beta_{i,j} (I_{\Pc})
\]
for all $i$ and $j$.
Thus we have the following immediately.

\begin{Proposition}
Let $\Pc$ and $\Pc'$ be lattice polytopes
such that $K[\Pc']$ is a combinatorial pure subring of $K[\Pc]$.
Suppose that $K[\Pc']$ is not isomorphic to a polynomial ring.
If $K[\Pc]$ has a $q$-linear resolution, then so does $K[\Pc']$

\end{Proposition}

Let $G$ be a finite simple graph on the vertex set $[n]$ with the edge set $E(G)$.
Given an edge $e = \{i,j\} \in E(G)$, we set $\rho(e) = \eb_i + \eb_j \in \ZZ^n$.
Here $\eb_i$ is the $i$-th unit vector in $\RR^n$.
The \textit{edge polytope} $\Pc_G$ of $G$ is the convex hull of 
\[
\{
\rho(e) : e \in E(G)
\}.
\]
Then the toric ring $K[\Pc_G]$ of $\Pc_G$ is isomorphic to the edge ring $K[G]$ of $G$.
A subgraph $G'$ of a graph $G$ is called an {\em induced subgraph} of $G$
if there exists $V \subset [n]$ such that $G'$ is a graph on the vertex set $V$
and the edge set
\[
\{ \{i,j\} \in E(G) : i,j  \in V\}.
\]
It is easy to see that, if $G'$ is an induced subgraph of a graph $G$,
then $\Pc_{G'}$ is a face of $\Pc_G$.
Thus we have the following.

\begin{Lemma}
\label{inducedsubgraph}
Let $G$ be a finite connected graph and let $G'$ be
an induced subgraph of $G$.
Suppose that $K[G']$ is not isomorphic to a polynomial ring.
If $K[G]$ has a $q$-linear resolution, then
so does $K[G']$.
\end{Lemma}

Let $\Pc \subset \RR^n$ be a lattice polytope of dimension $d$.
Then the {\em $\delta$-polynomial} (or {\em $h^*$-polynomial}) of $\Pc$ is a polynomial in $\lambda$ defined by
\[
\delta(\Pc, \lambda) = (1-\lambda)^{d+1} \left(  1+ \sum_{t=1}^\infty |t \Pc \cap \ZZ^n| \lambda^t \right),
\]
where $t \Pc =\{ t \ab : \ab \in \Pc\}$.
It is known that each coefficient of $\delta(\Pc, \lambda)$ is a nonnegative integer and the degree of $\delta(\Pc, \lambda)$
is at most $d$.
Let $\deg(\Pc) = \deg (\delta(\Pc, \lambda))$ and set $\codeg(\Pc) = d+1 - \deg(\Pc)$.
Then 
\[
\codeg (\Pc) = \min \{ r \in \ZZ_{>0} : \int (r \Pc) \cap \ZZ^n \neq \emptyset \},
\]
where $\int (r \Pc)$ is the relative interior of $r \Pc$ in $\RR^n$,
holds in general.
See, e.g., \cite[Part II]{HibiRedBook} for detailed information.
With some conditions, the degree of $\Pc$ gives a lower bound for $\reg (K[\Pc])$.
In fact, the following is known.

\begin{Lemma}[{\cite[Corollaries 3.2 and 3.4]{HMT}}]
\label{deglemma}
Let $G$ be a finite connected graph and let $G'$ be
a subgraph of $G$.
Then we have 
$\deg \Pc_{G'} \le \deg \Pc_{G} \le {\rm reg}(K[G])$.
\end{Lemma}

Recall that, if $K[G]$ has a $q$-linear resolution, then
 $\reg(K[G]) = q - 1$.
Hence, by Lemma~\ref{deglemma}, it turns out that $\deg (\Pc_G) < q$
is a necessary condition for $K[G]$ to have a $q$-linear resolution.

\section{Edge polytopes and edge rings}
\label{sec:edgepolyandring}

In this section, we study the degree of edge polytopes and 
the set of generators of the toric ideal of edge rings.
The results in the present section give necessary conditions for $K[G]$
to have a $q$-linear resolution.

\subsection{Dimension of edge polytopes}

First, we introduce a result on graphs with a small cyclotomic number since the cyclotomic number
is related to the codimension of $K[G]$.
Let $G$ be a connected graph with $n$ vertices and $m$ edges.
Then $c(G) = m-n+1$ is called the {\em cyclotomic number}
(or the {\em circuit rank}) of $G$.
Note that a connected graph $G$ satisfies $c(G)=0$ if and only if $G$ is a tree.
A connected graph with $c(G)=1,2,3$ is said to be {\em unicyclic},
{\em bicyclic}, and {\em tricyclic}, respectively.
It is known \cite{Ahrens} that the number of the cycles in a connected graph $G$ is
at least $c(G)$ and at most $2^{c(G)} -1$.
In particular, a connected graph is unicyclic if and only if
it has exactly one cycle.
The {\em edge subdivision operation} for an edge
$e=\{u,v\}$ of a graph $G$ is the deletion of $e$ from $G$ and the addition of two edges 
$\{u,w\}$ and $\{w,v\}$ along with the new vertex $w$.
A graph obtained from a graph $G$
 by a sequence of edge subdivision operations 
is called a {\em subdivision} of $G$.
Bicyclic and tricyclic graphs are characterized as follows.

\begin{Proposition}[{e.g., \cite{VM}}]
\label{bitri}
Let $G$ be a connected simple graph.
\begin{itemize}
\item[{\rm (a)}]
$G$ is bicyclic if and only if $G$ is obtained by adding some trees to a graph $G_0$,
where $G_0$ is a subdivision of one of the multigraphs in Fig.~1.

\item[{\rm (b)}]
$G$ is tricyclic if and only if $G$ is obtained by adding some trees to a graph $G_0$,
where $G_0$ is a subdivision of one of the  multigraphs in Fig.~2.
\end{itemize}
\end{Proposition}

\begin{figure}[h]
\begin{center}
\begin{tikzpicture} [scale=0.65]
\draw
(-8,0) circle(1)
(-6,0) circle (1);
\fill 
(-7,0) circle (2pt);

\draw
(-2,0) circle(1)
(2,0) circle(1)
(-1,0)--(1,0);
\fill
(-1,0) circle (2pt)
(1,0) circle (2pt);

\draw
(6,0) circle (1) 
(5,0)--(7,0);
\fill 
(5,0) circle (2pt)
(7,0) circle (2pt) ;
 \end{tikzpicture}
 \end{center}
\caption{Multigraphs for bicyclic graphs \cite{VM}}
\end{figure}
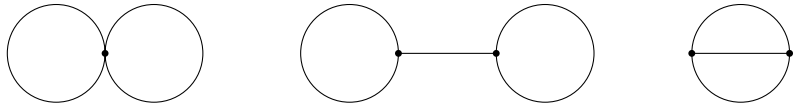
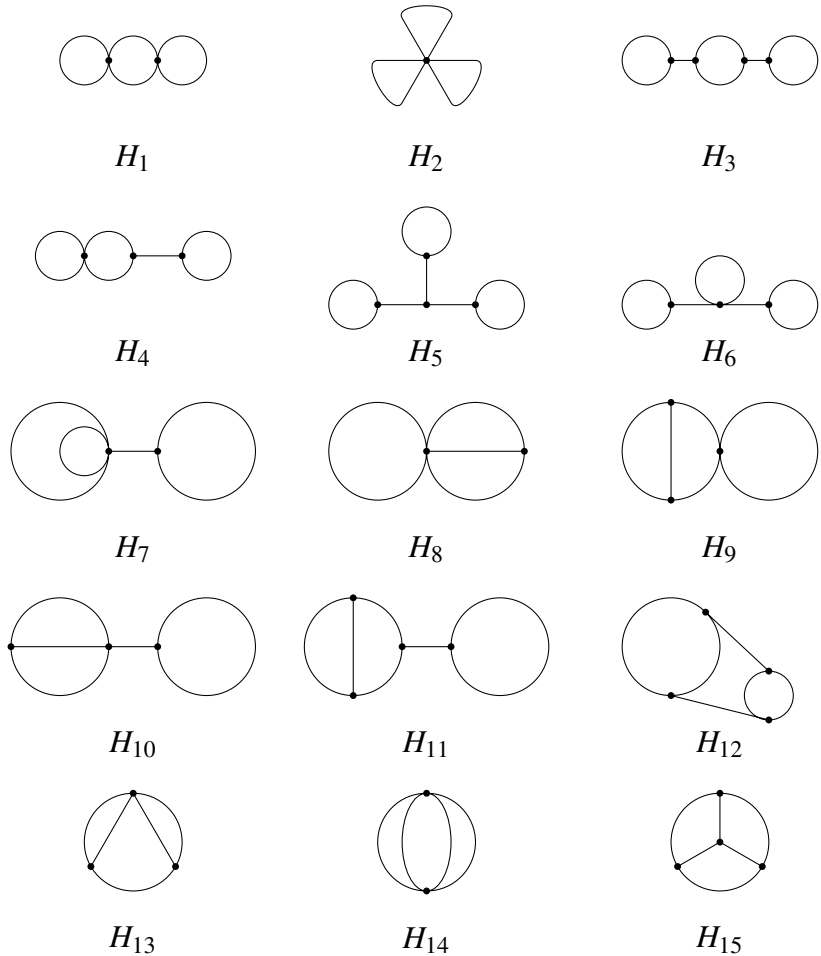
\begin{figure}
\begin{center}
\begin{tikzpicture} [scale=0.65]
\coordinate (v1) at (-6,6) node at (v1) {$H_{1}$};
 \draw
 (-6,8) circle (0.5)
 (-5,8) circle (0.5)
 (-7,8) circle (0.5);
 \fill
 (-5.5,8) circle (2pt)
 (-6.5,8) circle (2pt);
 
\coordinate (v2) at (0,6) node at (v2) {$H_{2}$};
\draw 
(0,8)--(-1/2,{8+sqrt(3)/2})
(0,8)--(1/2,{8+sqrt(3)/2})
(1/2,{8+sqrt(3)/2}) to [out=60,in=120]  (-1/2,{8+sqrt(3)/2})
(0,8)--(-1,8)
(0,8)--(-1/2,{8-sqrt(3)/2})
(-1,8) to [out=180,in=240]  (-1/2,{8-sqrt(3)/2})
(0,8)--(1,8)
(0,8)--(1/2,{8-sqrt(3)/2})
(1,8) to [out=360,in=300]  (1/2,{8-sqrt(3)/2});
\fill
(0,8) circle (2pt);

\coordinate (v3) at (6,6) node at (v3) {$H_{3}$};
\draw
 (6,8) circle (0.5)
 (4.5,8) circle (0.5)
 (7.5,8) circle (0.5)
 (5,8)--(5.5,8)
 (6.5,8)--(7,8);
 \fill
 (5,8) circle (2pt)
 (5.5,8) circle (2pt)
 (6.5,8) circle (2pt)
 (7,8) circle (2pt);
 
\coordinate (v4) at (-6,2) node at (v4) {$H_{4}$};
 \draw
 (-6,4)--(-5,4)
 (-4.5,4) circle (0.5)
 (-7.5,4) circle (0.5)
 (-6.5,4) circle (0.5);
 \fill
 (-7,4) circle (2pt)
 (-6,4) circle (2pt)
 (-5,4) circle (2pt);
 
 \coordinate (v5) at (0,2) node at (v5) {$H_{5}$};
 \draw
 (0,4.5) circle (0.5)
 (-1.5,3) circle (0.5)
 (1.5,3) circle (0.5)
 (0,3)--(0,4)
 (-1,3)--(1,3);
 \fill
 (0,3) circle (2pt)
 (0,4) circle (2pt)
 (-1,3) circle (2pt)
 (1,3) circle (2pt);
 
 \coordinate (v6) at (6,2) node at (v6) {$H_{6}$};
 \draw
 (6,3.5) circle (0.5)
 (4.5,3) circle (0.5)
 (7.5,3) circle (0.5)
 (5,3)--(7,3);
 \fill
 (6,3) circle (2pt)
 (5,3) circle (2pt)
 (7,3) circle (2pt);
 
  \coordinate (v7) at (-6,-2) node at (v7) {$H_{7}$};
 \draw
 (-6.5,0)--(-5.5,0)
 (-4.5,0) circle (1)
 (-7.5,0) circle (1)
 (-7,0) circle (0.5);
 \fill
 (-6.5,0) circle (2pt)
 (-5.5,0) circle (2pt);
 
\coordinate (v8) at (0,-2) node at (v8) {$H_{8}$};
 \draw
 (-1,0) circle (1)
 (1,0) circle (1)
 (0,0)--(2,0);
 \fill
 (0,0) circle (2pt)
 (2,0) circle (2pt);
 
 \coordinate (v9) at (6,-2) node at (v9) {$H_{9}$};
 \draw
 (5,0) circle (1)
 (7,0) circle (1)
 (5,1)--(5,-1);
 \fill
 (6,0) circle (2pt) 
 (5,1) circle (2pt)
 (5,-1) circle (2pt);
 
 \coordinate (v10) at (-6,-6) node at (v10) {$H_{10}$};
  \draw
 (-8.5,-4)--(-5.5,-4)
 (-4.5,-4) circle (1)
 (-7.5,-4) circle (1);
 \fill
 (-6.5,-4) circle (2pt)
 (-5.5,-4) circle (2pt)
 (-8.5,-4) circle (2pt);
 
  \coordinate (v11) at (0,-6) node at (v11) {$H_{11}$};
  \draw
 (-0.5,-4)--(0.5,-4)
 (-1.5,-3)--(-1.5,-5)
 (1.5,-4) circle (1)
 (-1.5,-4) circle (1);
 \fill
 (-0.5,-4) circle (2pt)
 (0.5,-4) circle (2pt)
 (-1.5,-3) circle (2pt)
 (-1.5,-5) circle (2pt);
 
 \coordinate (v12) at (6,-6) node at (v12) {$H_{12}$};
\draw
(5,-4) circle(1)
(7,-5) circle(0.5)
({5+sqrt(2)/2},{-4+sqrt(2)/2})--(7,-4.5)
(5,-5)--(7,-5.5);
\fill
({5+sqrt(2)/2},{-4+sqrt(2)/2}) circle (2pt)
(7,-4.5) circle (2pt)
(5,-5) circle (2pt)
(7,-5.5)circle (2pt);

\coordinate (v13) at (-6,-10) node at (v13) {$H_{13}$};
\draw
(-6,-8) circle (1)
(-6,-7)--({-6-sqrt(3)/2},-8.5)
(-6,-7)--({-6+sqrt(3)/2},-8.5);
\fill
(-6,-7) circle (2pt)
({-6-sqrt(3)/2},-8.5) circle (2pt)
({-6+sqrt(3)/2},-8.5) circle(2pt);

\coordinate (v14) at (0,-10) node at (v14) {$H_{14}$};
\draw
(0,-8) circle (1)
(0,-8) circle [x radius=0.5,y radius=1];
\fill
(0,-7) circle (2pt)
(0,-9) circle (2pt);

\coordinate (v15) at (6,-10) node at (v15) {$H_{15}$};
\draw
(6,-8) circle (1)
(6,-8)--(6,-7)
(6,-8)--({6-sqrt(3)/2},-8.5)
(6,-8)--({6+sqrt(3)/2},-8.5);
\fill
(6,-7) circle (2pt)
(6,-8) circle (2pt)
({6-sqrt(3)/2},-8.5) circle (2pt)
({6+sqrt(3)/2},-8.5) circle (2pt);

 \end{tikzpicture}  
 \end{center}
\caption{Multigraphs for tricyclic graphs \cite{VM}}
\end{figure}

The following proposition is essentially proved in \cite[Theorem~2.5]{GKS}.

\begin{Proposition}
Let $G$ be a finite graph with $n$ vertices.
Then 
\[
\dim \Pc_G = n-r(G) -1
,\]
where $r(G)$ is the number of connected components of $G$ that are bipartite.
\end{Proposition}

In particular, if $G$ is a connected graph with $n$ vertices and $m$ edges, then
\begin{eqnarray*}
\dim (K[G]) &=& \dim \Pc_G  +1=n-r(G),\\
{\rm codim}(K[G]) &=& m-  \dim (K[G]) =m-n+r(G) =
\left\{
\begin{array}{cc}
c(G) & G \mbox{ is bipartite},\\
c(G)-1 & \mbox{otherwise}.
\end{array}
\right.
\end{eqnarray*}
Note that the following conditions are equivalent:
\begin{itemize}
\item[(i)]
$K[G]$ is isomorphic to a polynomial ring;
\item[(ii)]
${\rm codim}(K[G])=0$;
\item[(iii)]
$\Pc_G$ is a simplex.
\end{itemize}
Graphs satisfying such conditions are completely classified.

\begin{Proposition}[{\cite[Lemmas 5.5 and 5.6]{binomialideals}}]
\label{simplex}
Let $G$ be a finite connected graph with $n$ vertices and $m$ edges.
Then $K[G]$ is isomorphic to a polynomial ring
 if and only if $G$ satisfies one of the following{\rm :}
\begin{itemize}
\item[{\rm (a)}]
$m=n-1$ and $G$ is a tree{\rm ;}
\item[{\rm (b)}]
$m=n$ and $G$ has exactly one cycle $C$.
In addition, $C$ is an odd cycle.
\end{itemize}
\end{Proposition}

\subsection{Degree of edge polytopes}

In the present subsection, we give several sufficient conditions for graphs $G$
to satisfy $\deg (\Pc_G)\ge q$.
The following two lemmas are given in \cite{Tbipartite}.

\begin{Lemma}[{\cite[Lemmas 2.1 and 2.2]{Tbipartite}}]
\label{T2lemma}
Let $q  \ge 3$ be an integer.
If $G$ has two even cycles $C_1$ and $C_2$ of length $2q$
having at most one common vertex,
then we have $\deg \Pc_{G} \ge q$.
\end{Lemma}

\begin{Lemma}[{\cite[Lemmas 2.3 and 2.4]{Tbipartite}}]
\label{T3lemma}
Let $q  \ge 3$ be an integer.
If $G$ has a bipartite subgraph obtained by adding a path 
to an even cycle, and if the length of every cycle in $G$ is $2q$,
then we have $\deg \Pc_{G} \ge q$.
\end{Lemma}

We now give two important lemmas about sufficient conditions for graphs $G$
to satisfy $\deg (\Pc_G)\ge q$.

\begin{Lemma}
\label{degqqq}
Let $q  \ge 3$ be an integer.
If $G$ has an even cycle of length $ > 2q$, then we have 
$\deg (\Pc_G)\ge q$.
\end{Lemma}

\begin{proof}
Let $C = (i_1,i_2,\dots,i_{2r})$ be an even cycle in $G$ of length $2r > 2q$.
Then the dimension of $\Pc_C$ is $2r-2$. 
Since
\[
\frac{1}{2} \sum_{e \in E(C)} \rho(e)
=\eb_{i_1} + \cdots + \eb_{i_{2r}} \in \int (r \Pc_C)\cap \ZZ^{2r},
\]
we have $\codeg (\Pc_C) \le r$.
By Lemma \ref{deglemma}, $\deg (\Pc_G) \ge 2r-2+1- r =r-1 \ge q$.
\end{proof}

A vertex $v$ of a connected graph $G$ is called a {\em cut vertex}
if the graph obtained by the removal of $v$ from $G$ is disconnected.
Given a graph $G$, a {\em block} of $G$ is a maximal connected subgraph of $G$ with no cut vertices.

\begin{Lemma}
\label{codimension2lemma}
Let $q  \ge 3$ be an integer.
Let $G$ be a connected tricyclic graph having 
at least one even cycle.
If the length of every even cycle in $G$ is $2q$, then we have 
$\deg (\Pc_G)\ge q$.
\end{Lemma}

\begin{proof}
By Proposition~\ref{bitri} (b), there exists a subgraph $G_0$ of $G$
that is a subdivision of a graph $H$ in Fig.~2.
Let $C$ and $C'$ be even cycles in $G_0$.
By Lemma~\ref{T2lemma}, 
we may assume that $C$ and $C$ have at least two common vertices.
If $C \cup C'$ is a bicyclic graph, then $C \cup C'$
is a graph appearing in Lemma \ref{T3lemma} and hence $\deg (\Pc_G) \ge q$.
Hence we may assume that
\begin{itemize}
\item[($*$)]
If $C$ and $C'$ are even cycles in $G_0$, then $C$ and $C'$ have at least two common vertices,
and $C \cup C'$ is not bicyclic.
\end{itemize}


\noindent
{\bf Case 1.} ($H=H_1,\ldots,H_7$.)
Then $G$ has exactly three cycles.
By condition ($*$), 
we may assume that $G$ has one even cycle $C_1= (i_1, i_2,\ldots, i_{2q})$ of length $2q$ and
two odd cycles $C_2=(j_1, j_2,\ldots,j_{2r+1})$ and $C_3= (k_1, k_2,\ldots,k_{2s+1})$.
Let $H' =C_1 \cup C_2 \cup C_3 $.
Then the dimension of $\Pc_{H'}$ is 
$2q+2r+2s-t$ for some $t \in \{0,1\}$.
Since all of
\begin{eqnarray*}
\alpha&=&\frac{1}{2} \sum_{e \in E(C_1)} \rho(e) = \eb_{i_1} + \cdots + \eb_{i_{2q}}\\
\beta&=&\frac{1}{2} \sum_{e \in E(C_2)} \rho(e) = \eb_{j_1} + \cdots + \eb_{j_{2r+1}}\\
\gamma&=&\frac{1}{2} \sum_{e \in E(C_3)} \rho(e) = \eb_{k_1} + \cdots + \eb_{k_{2s+1}}
\end{eqnarray*}
are integer vectors, we have
\[
\alpha + \beta + \gamma \in \int ((q+r+s+1) \Pc_{H'})\cap \ZZ^{|V(H')|}.
\]
Hence $\codeg (\Pc_{H'}) \le q+r+s+1$.
By Lemma \ref{deglemma}, 
$$\deg (\Pc_G) \ge 2q+2r+2s-t +1- q-r-s-1 = q+r+s-t\ge q.$$


\noindent
{\bf Case 2.} ($H=H_8, H_9, H_{10}, H_{11}$.)
Then $G_0$ has exactly two blocks $B_1$ and $B_2$
where $B_1$ is a cycle and $B_2$ is obtained by adding a path 
to a cycle.
Note that $B_2$ contains an even cycle $C_1=(e_1,\ldots,e_{2q})$.
By Lemmas \ref{T2lemma} and \ref{T3lemma},
we may assume that $B_1$ is an odd cycle and $B_2$ is not bipartite.
Let $C_2$ be one of two odd cycles in $B_2$.
Let $G'$ be a subgraph $B_1 \cup B_2$ of $G$.
Then the dimension of $\Pc_{G'}$ is
$2q + |E(B_1 \cup C_2)| - |E(C_1 \cap C_2)|-2-t$ for some $t \in \{0,1\}$.
Since all of
\begin{eqnarray*}
\alpha&=&\frac{1}{2} \sum_{e \in E(B_1 \cup C_2)} \rho(e)\\
\beta&=&\frac{1}{3} \sum_{i=1}^{q} \rho(e_{2i-1})
+
\frac{2}{3} \sum_{i=1}^{q} \rho(e_{2i})\\
\gamma&=&
\sum_{\substack{1 \le i \le q \\ e_{2i} \in E(C_2)}} \rho(e_{2i})
\end{eqnarray*}
are integer vectors, and $1/2 + 2/3 -1 = 1/6 >0$,
\[
\alpha + \beta -\gamma \in \int (r \Pc_{G'})\cap \ZZ^{|V(G')|},
\]
where $r=q+|E(B_1 \cup C_2)|/2-|E(C_1 \cap C_2)| $.
Hence we have $\codeg (\Pc_{G'}) \le r$.
By Lemma~\ref{deglemma}, 
\begin{eqnarray*}
\deg (\Pc_G) &\ge& 2q + |E(B_1 \cup C_2)| - |E(C_1 \cap C_2)|-2-t+1
- r\\
 &=&
q + |E(B_1 \cup C_2)|/2-1-t\\
& \ge& q.
\end{eqnarray*}


\noindent
{\bf Case 3.} ($H=H_{12}, H_{13}, H_{14}, H_{15}$.)

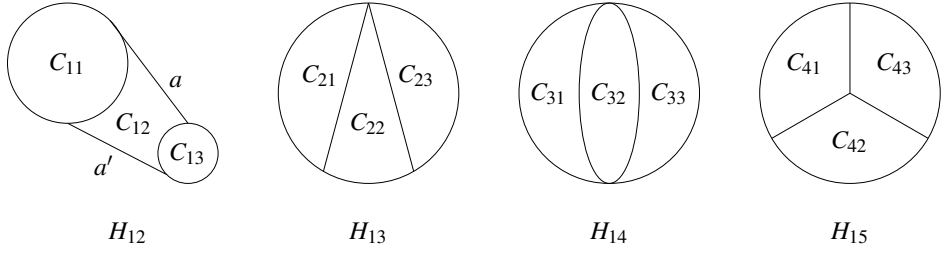
\begin{figure}
\scalebox{0.8}{
\begin{tikzpicture}
\coordinate (c100) at (-3.2,0.2) node at (c100) {$a$};
\coordinate (c101) at (-4.4,-1.2) node at (c101) {$a'$};

\draw(-5,0.5) circle(1);
\draw (-3,-1) circle(0.5);
\draw (-5,-0.5)--({-3-sqrt(2)/4},{-1-sqrt(2)/4});
\draw ({-5+sqrt(2)/2},{0.5+sqrt(2)/2})--(-3,-0.5);
\coordinate (c1) at (-5,0.5) node at (c1) {$C_{11}$};
\coordinate (c2) at (-3.9,-0.5) node at (c2) {$C_{12}$};
\coordinate (c3) at (-3,-1) node at (c3) {$C_{13}$};
\coordinate (v1) at (-4,-2) node at (v1) [below] {$H_{12}$};

\draw(0,0) circle(1.5);
\draw (0,1.5)--(-3/4,{-3*sqrt(3)/4});
\draw (0,1.5)--(3/4,{-3*sqrt(3)/4});
\coordinate (c10) at (-0.8,0.3) node at (c10) {$C_{21}$};
\coordinate (c11) at (0,-0.5) node at (c11) {$C_{22}$};
\coordinate (c12) at (0.8,0.3) node at (c12) {$C_{23}$};
\coordinate (v2) at (0,-2) node at (v2) [below] {$H_{13}$};

\draw (4,0) circle [x radius=0.5,y radius=1.5];
\draw (4,0) circle (1.5);
\coordinate (c4) at (3,0) node at (c4) {$C_{31}$};
\coordinate (c5) at (4,0) node at (c5) {$C_{32}$};
\coordinate (c6) at (5,0) node at (c6) {$C_{33}$};
\coordinate (v3) at  (4,-2) node at (v3) [below] {$H_{14}$};

\draw (8,0) circle (1.5);
\draw (8,1.5)--(8,0);
\draw ({8-1.5*sqrt(3)/2},-3/4)--(8,0);
\draw ({8+1.5*sqrt(3)/2},-3/4)--(8,0);
\coordinate (c7) at (7.25,0.5) node at (c7) {$C_{41}$};
\coordinate (c8) at (8,-0.8) node at (c8) {$C_{42}$};
\coordinate (c9) at (8.75,0.5) node at (c9) {$C_{43}$};
\coordinate (v4) at (8,-2) node at (v4) [below] {$H_{15}$};
 \end{tikzpicture}
}
\caption{$H_{12}, \dots, H_{15}$}
\label{four cases}
\end{figure}

Note that, if $C$ and $C'$ are odd cycles in $G_0$ having at least one common edge
such that 
$C \cup C'$ is a bicyclic graph, then $C \cup C'$ contains an even cycle.
Hence by condition ($*$), it follows that any $C_{ij}$ in Fig.~\ref{four cases} is an odd cycle
except for $C_{12}$, $C_{22}$, and $C_{32}$.
If $C_{i2}$ is an even cycle, then we replace $C_{i2}$ with the odd cycle
whose edge set is $E(C_{i1}) \cup E(C_{i2}) \setminus (E(C_{i1}) \cap E(C_{i2}))$.
Thus we may assume that any $C_{ij}$ in Fig.~\ref{four cases} is an odd cycle.
%

\bigskip

\noindent
{\bf Case 3.1.} ($H=H_{12}$, $a+a'$ is odd.)

We may assume that $a$ is odd and $a'$ is even.
Let $a=2s+1$ and $a'=2t$.
Given positive integers $k$, $\ell$, $k'$, and $\ell'$ with  
\[
2k+1+2\ell+a+a'=2k'+2\ell'+1+a+a'= 2q,
\]
 let $G_0$ be a graph on the vertex set 
$[m'-2]$ with the edge set 
$E(G_0)=\{e_1,\ldots, e_{m'} \}$, where $m'=2q+2k'+2\ell'+1$ and 
	\[
	e_i=\begin{cases}
		\{i,i+1\} & 1 \leq i \leq 2q-1,\\
		\{1,i \} & i=2q,2q+1,\\
		\{i-1,i\} & 2q+2 \leq i \leq 2q+2k'-1,\\
		\{2q+2k'-1,2k+2\} & i=2q+2k',\\
		\{2k+1+a+1,2q+2k'\} & i=2q+2k'+1,\\   
		\{i-2,i-1\} & 2q+2k'+2 \leq i \leq 2q+2k'+2\ell',\\
		\{2q+2k'+2\ell'-1,2q-a'+1\} & i=2q+2k'+2\ell'+1 .
	\end{cases}
	\]
See Fig.~\ref{case3.1}.
\begin{figure}[h]
\begin{center}
\scalebox{0.7}{
 \begin{tikzpicture}
        
 \draw[dotted][line width=0.5pt] (-3,3) arc (90:270:3); 
 \draw[line width=0.5pt] (-3,3) arc (90:160:3);
 \draw[line width=0.5pt] (-3,-3) arc (270:200:3);
  \draw[dotted][line width=0.5pt] (3,-3) arc (270:450:3); 
 \draw[line width=0.5pt] (3,-3) arc (270:340:3);
 \draw[line width=0.5pt] (3,3) arc (90:20:3);
 
\draw[dotted]  (3,3) rectangle (-3,-3);
\draw (-0.5,3)--(-3,3)--(-3,0.5);
 \draw (-3,-0.5)--(-3,-3)--(-0.5,-3);
 \draw (0.5,-3)--(3,-3)--(3,-0.5);
 \draw (3,0.5)--(3,3)--(0.5,3);
 
\coordinate (v1) at (-3,3) node at (v1) [above] {$1$};
\coordinate (v2) at (-3,2) node at (v2) [left] {$2$};
\coordinate (v3) at (-3,-2) node at (v3) [left] {$2k+1$};
\coordinate (v4) at (-3,-3) node at (v4) [below] {$2k+2$};
\coordinate (v5) at (-1.5,-3) node at (v5) [below] {$2k+3$}; 
\coordinate (v6) at (1.5,-3) node at ( v6) [below] {$2k+1+a$};
\coordinate (v7) at (3,-3) node at (v7) [below right] {$2k+1+a+1$};
\coordinate (v8) at (3,-2) node at (v8) [left] {$2q+2k'$};
\coordinate (v9) at (3,2)  node at (v9) [left] {$2q+2k'+2\ell'-1$};
\coordinate (v10) at (3,3) node at (v10) [above] {$2q-a'+1$};
 \coordinate (v11) at (1.5,3) node at (v11) [above left] {$2q-a'+2$};
 \coordinate (v12) at (-1.5,3) node at (v12) [above] {$2q$};
\coordinate (v13) at (-{9/2},{3*sqrt(3)/2}) node at (v13) [above left] {$2q+1$};
\coordinate (v14) at ({-3*sqrt(3)/2-3},3/2) node at (v14) [above left] {$2q+2$}; 
\coordinate (v15) at ({-3*sqrt(3)/2-3},-3/2) node at (v15) [below left] {$2q+2k'-2$};
\coordinate (v16) at ({-9/2},{-3*sqrt(3)/2}) node at (v16) [below left] {$2q+2k'-1$};
\coordinate (v17) at ({9/2},{-3*sqrt(3)/2}) node at (v17) [right] {$2k+1+a+2$};
\coordinate (v18) at ({3*sqrt(3)/2+3},{-3/2}) node at (v18) [right] {$2k+1+a+3$};
\coordinate (v20) at ({9/2},{3*sqrt(3)/2}) node at (v20) [above right] {$2k+1+a+2\ell$};
\coordinate (v19) at ({3+3*sqrt(3)/2},3/2) node at (v19) [above right] {$2k+1+a+2\ell-1$};
\coordinate (v21) at (-3,1) node at (v21) [left] {$3$};
\coordinate (v22) at (-3,-1) node at (v22) [left] {$2k$};
\coordinate (v25) at (3,1) node at (v25) [left] {$2q+2k'+2\ell'-2$};
\coordinate (v26) at (3,-1) node at (v26) [left] {$2q+2k'+1$};
 
\fill
(v1) circle (2pt)
(v2) circle (2pt)
(v3) circle (2pt)
(v4) circle (2pt)
(v5) circle (2pt)
(v6) circle (2pt)
(v7)circle (2pt)
(v8) circle (2pt)
(v9) circle (2pt)
(v10)circle (2pt)
(v11) circle (2pt)
(v12) circle (2pt)
(v13) circle (2pt)
(v14)circle (2pt)
(v15) circle (2pt)
(v16) circle (2pt)
(v17) circle (2pt)
(v18) circle (2pt)
(v19) circle (2pt)
(v20) circle (2pt)
(v21)circle (2pt)
(v22) circle (2pt)
(v25) circle (2pt)
(v26) circle (2pt)
;

 \end{tikzpicture}
 }
\scalebox{0.7}{
 \begin{tikzpicture}
        
 \draw[dotted][line width=0.5pt] (-3,3) arc (90:270:3); 
 \draw[line width=0.5pt] (-3,3) arc (90:160:3);
 \draw[line width=0.5pt] (-3,-3) arc (270:200:3);
  \draw[dotted][line width=0.5pt] (3,-3) arc (270:450:3); 
 \draw[line width=0.5pt] (3,-3) arc (270:340:3);
 \draw[line width=0.5pt] (3,3) arc (90:20:3);
 
\draw[dotted]  (3,3) rectangle (-3,-3);
\draw (-0.5,3)--(-3,3)--(-3,0.5);
 \draw (-3,-0.5)--(-3,-3)--(-0.5,-3);
 \draw (0.5,-3)--(3,-3)--(3,-0.5);
 \draw (3,0.5)--(3,3)--(0.5,3);

\coordinate (e1) at (-3,2.5) node at (e1) {$e_1$};
\coordinate (e2) at (-3,-2.5) node at (e2) {$e_{2k+1}$};
\coordinate (e3) at (-2.25,-3.2) node at (e3)  {$e_{2k+2}$};
\coordinate (e4) at (2.25,-3.2) node at (e4)  {$e_{2k+1+a}$};
\coordinate (e5) at (3,-2.5) node at (e5)  {$e_{2q+2k'+1}$};
\coordinate (e6) at (3,2.5) node at (e6)  {$e_{2q+2k'+2\ell'+1}$};
\coordinate (e7) at (2.25,3.2) node at (e7)  {$e_{2q-a'+1}$}; 
\coordinate (e8) at (-2.25,3.2) node at (e8)  {$e_{2q}$};
\coordinate (e9) at (-4,3) node at (e9)  {$e_{2q+1}$};
\coordinate (e10) at (-5.1,2.1) node at (e10)  {$e_{2q+2}$};
\coordinate (e11) at (-5.1,-2.1) node at (e11)  {$e_{2q+2k'-1}$};
\coordinate (e12) at (-4,-3.2) node at (e12)  {$e_{2q+2k'}$};
\coordinate (e13) at (4,-3.2) node at (e13)  {$e_{2k+1+a+1}$};
\coordinate (e14) at (5.1,-2.1) node at (e14)  {$e_{2k+1+a+2}$};
\coordinate (e15) at (5.1,2.25) node at (e15) [ right] {$e_{2k+1+a+2\ell-1}$};
\coordinate (e16) at (3.2,3) node at (e16) [right] {$e_{2k+1+a+2\ell}$};
\coordinate (e17) at (-3,1.5) node at (e17)  {$e_2$};
\coordinate (e18) at (-3,-1.5) node at (e18) {$e_{2k}$};
\coordinate (e19) at (3,-1.5) node at (e19) {$e_{2q+2k'+2}$};
\coordinate (e20) at (3,1.5) node at (e20)  {$e_{2q+2k'+2\ell'}$};

\fill
(v1) circle (2pt)
(v2) circle (2pt)
(v3) circle (2pt)
(v4) circle (2pt)
(v5) circle (2pt)
(v6) circle (2pt)
(v7)circle (2pt)
(v8) circle (2pt)
(v9) circle (2pt)
(v10)circle (2pt)
(v11) circle (2pt)
(v12) circle (2pt)
(v13) circle (2pt)
(v14)circle (2pt)
(v15) circle (2pt)
(v16) circle (2pt)
(v17) circle (2pt)
(v18) circle (2pt)
(v19) circle (2pt)
(v20) circle (2pt)
(v21)circle (2pt)
(v22) circle (2pt)
(v25) circle (2pt)
(v26) circle (2pt)
;

 \end{tikzpicture}
 }
 \end{center}
\caption{Case 3.1}
\label{case3.1}
\end{figure}
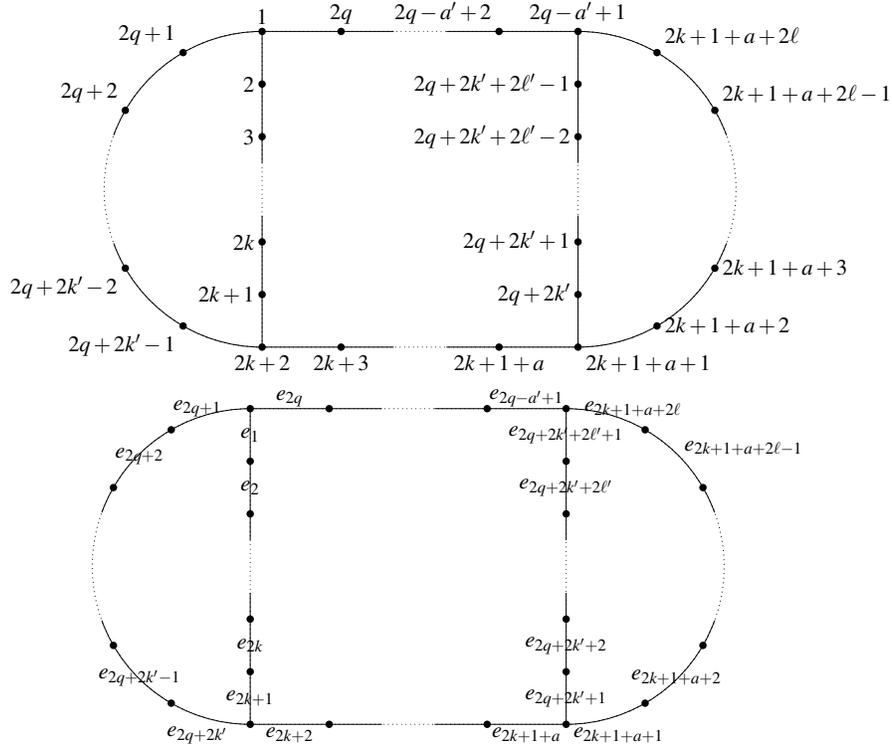
Then one has $\dim({\mathcal{P}_{G_0}})=2q+2k'+2\ell'-1-1=2q+2k+2\ell-2$.
Since
	\begin{eqnarray*}
	&	&\sum_{i=1}^{2k+1} \dfrac{1}{2} \rho(e_{i}) +\sum_{i=1}^{2\ell} \dfrac{1}{2} \rho(e_{2k+1+a+i}) +\sum_{i=1}^{k'} \left(\dfrac{2}{3} \rho(e_{{2q+2i-1}})+ \dfrac{1}{3}\rho(e_{2q+2i})\right)\\
		&+&\sum_{i=1}^{\ell'} \left(\dfrac{1}{3} \rho(e_{{2q+2k'+2i-1}}) + \dfrac{2}{3}\rho(e_{2q+2k'+2i})\right)+\dfrac{1}{3} \rho(e_{{2q+2k'+2\ell'+1}})\\
		&+&\sum_{i=1}^{s} \left(\dfrac{1}{6} \rho(e_{2k+1+{2i-1}}) + \dfrac{5}{6}\rho(e_{2k+1+2i})\right)+\dfrac{1}{6} \rho(e_{2k+1+2s+1})\\	
		&+& \sum_{i=1}^{t} \left(\dfrac{1}{6} \rho(e_{2k+1+a+2\ell+{2i-1}}) + \dfrac{5}{6}\rho(e_{2k+1+a+2\ell+2i})\right)\\
		&= &2 {\eb_1}+{\eb_2}+\cdots+\eb_{m'-2} } \in {\rm int}((k+\ell+k'+\ell'+s+t+1) \mathcal{P}_{G_0}) \cap{\mathbb{Z}^{m'-2},
	\end{eqnarray*}
	we obtain  ${\rm codeg}(\mathcal{P}_{G_0}) \leq q+k+\ell$.
	Hence it follows that
	\[
	{\rm deg}(\mathcal{P}_{G_0})
 \geq 2q+2k+2\ell-2+1-q-k-\ell=q+k+\ell-1 \geq q.
	\]


\noindent
{\bf Case 3.2.} ($H=H_{12}$, $a+a'$ is even.)

Given positive integers $k$, $\ell$, $k'$, and $\ell'$ with  
\[
2k+1+2\ell+1+a+a'=2k'+2\ell'+a+a'= 2q,
\]
let $G_0$ be a graph on the vertex set $[m'-2]$ 
with the edge set $E(G_0)=\{e_1,\ldots, e_{m'} \}$, where $m'=2q+2k'+2\ell'$ and 
	\[
	e_i=\begin{cases}
		\{i,i+1\} & 1 \leq i \leq 2q-1,\\
		\{1,i \} & i=2q,2q+1\\
		\{i-1,i\} & 2q+2 \leq i \leq 2q+2k'-1,\\
		\{2q+2k'-1,2k+2\} & i=2q+2k',\\
		\{2k+1+a+1,2q+2k'\} & i=2q+2k'+1,\\   
		\{i-2,i-1\} & 2q+2k'+2 \leq i \leq 2q+2k'+2\ell'-1,\\
		\{2q+2k'+2\ell'-2,2q-a'+1\} & i=2q+2k'+2\ell' .
	\end{cases}
	\]
See Fig.~\ref{case3.2}.
\begin{figure}[h]
\begin{center}
\scalebox{0.7}{
 \begin{tikzpicture}
        
 \draw[dotted][line width=0.5pt] (-3,3) arc (90:270:3); 
 \draw[line width=0.5pt] (-3,3) arc (90:160:3);
 \draw[line width=0.5pt] (-3,-3) arc (270:200:3);
  \draw[dotted][line width=0.5pt] (3,-3) arc (270:450:3); 
 \draw[line width=0.5pt] (3,-3) arc (270:340:3);
 \draw[line width=0.5pt] (3,3) arc (90:20:3);
 
\draw[dotted]  (3,3) rectangle (-3,-3);
\draw (-0.5,3)--(-3,3)--(-3,0.5);
 \draw (-3,-0.5)--(-3,-3)--(-0.5,-3);
 \draw (0.5,-3)--(3,-3)--(3,-0.5);
 \draw (3,0.5)--(3,3)--(0.5,3);
 
\coordinate (v1) at (-3,3) node at (v1) [above] {$1$};
\coordinate (v2) at (-3,2) node at (v2) [left] {$2$};
\coordinate (v3) at (-3,-2) node at (v3) [left] {$2k+1$};
\coordinate (v4) at (-3,-3) node at (v4) [below] {$2k+2$};
\coordinate (v5) at (-1.5,-3) node at (v5) [below] {$2k+3$}; 
\coordinate (v6) at (1.5,-3) node at ( v6) [below] {$2k+1+a$};
\coordinate (v7) at (3,-3) node at (v7) [below right] {$2k+1+a+1$};
\coordinate (v8) at (3,-2) node at (v8) [left] {$2q+2k'$};
\coordinate (v9) at (3,2)  node at (v9) [left] {$2q+2k'+2\ell'-2$};
\coordinate (v10) at (3,3) node at (v10) [above] {$2q-a'+1$};
 \coordinate (v11) at (1.5,3) node at (v11) [above left] {$2q-a'+2$};
 \coordinate (v12) at (-1.5,3) node at (v12) [above] {$2q$};
\coordinate (v13) at (-{9/2},{3*sqrt(3)/2}) node at (v13) [above left] {$2q+1$};
\coordinate (v14) at ({-3*sqrt(3)/2-3},3/2) node at (v14) [above left] {$2q+2$}; 
\coordinate (v15) at ({-3*sqrt(3)/2-3},-3/2) node at (v15) [below left] {$2q+2k'-2$};
\coordinate (v16) at ({-9/2},{-3*sqrt(3)/2}) node at (v16) [below left] {$2q+2k'-1$};
\coordinate (v17) at ({9/2},{-3*sqrt(3)/2}) node at (v17) [right] {$2k+1+a+2$};
\coordinate (v18) at ({3*sqrt(3)/2+3},{-3/2}) node at (v18) [right] {$2k+1+a+3$};
\coordinate (v20) at ({9/2},{3*sqrt(3)/2}) node at (v20) [above right] {$2k+1+a+2\ell+1$};
\coordinate (v19) at ({3+3*sqrt(3)/2},3/2) node at (v19) [above right] {$2k+1+a+2\ell$};
\coordinate (v21) at (-3,1) node at (v21) [left] {$3$};
\coordinate (v22) at (-3,-1) node at (v22) [left] {$2k$};
\coordinate (v25) at (3,1) node at (v25) [left] {$2q+2k'+2\ell'-3$};
\coordinate (v26) at (3,-1) node at (v26) [left] {$2q+2k'+1$};

\fill
(v1) circle (2pt)
(v2) circle (2pt)
(v3) circle (2pt)
(v4) circle (2pt)
(v5) circle (2pt)
(v6) circle (2pt)
(v7)circle (2pt)
(v8) circle (2pt)
(v9) circle (2pt)
(v10)circle (2pt)
(v11) circle (2pt)
(v12) circle (2pt)
(v13) circle (2pt)
(v14)circle (2pt)
(v15) circle (2pt)
(v16) circle (2pt)
(v17) circle (2pt)
(v18) circle (2pt)
(v19) circle (2pt)
(v20) circle (2pt)
(v21)circle (2pt)
(v22) circle (2pt)
(v25) circle (2pt)
(v26) circle (2pt)
;

 \end{tikzpicture}
 }
\scalebox{0.7}{
 \begin{tikzpicture}
        
 \draw[dotted][line width=0.5pt] (-3,3) arc (90:270:3); 
 \draw[line width=0.5pt] (-3,3) arc (90:160:3);
 \draw[line width=0.5pt] (-3,-3) arc (270:200:3);
  \draw[dotted][line width=0.5pt] (3,-3) arc (270:450:3); 
 \draw[line width=0.5pt] (3,-3) arc (270:340:3);
 \draw[line width=0.5pt] (3,3) arc (90:20:3);
 
\draw[dotted]  (3,3) rectangle (-3,-3);
\draw (-0.5,3)--(-3,3)--(-3,0.5);
 \draw (-3,-0.5)--(-3,-3)--(-0.5,-3);
 \draw (0.5,-3)--(3,-3)--(3,-0.5);
 \draw (3,0.5)--(3,3)--(0.5,3);

\coordinate (e1) at (-3,2.5) node at (e1) {$e_1$};
\coordinate (e2) at (-3,-2.5) node at (e2) {$e_{2k+1}$};
\coordinate (e3) at (-2.25,-3.2) node at (e3)  {$e_{2k+2}$};
\coordinate (e4) at (2.25,-3.2) node at (e4)  {$e_{2k+1+a}$};
\coordinate (e5) at (3,-2.5) node at (e5)  {$e_{2q+2k'+1}$};
\coordinate (e6) at (3,2.5) node at (e6)  {$e_{2q+2k'+2\ell'}$};
\coordinate (e7) at (2.25,3.2) node at (e7)  {$e_{2q-a'+1}$}; 
\coordinate (e8) at (-2.25,3.2) node at (e8)  {$e_{2q}$};
\coordinate (e9) at (-4,3) node at (e9)  {$e_{2q+1}$};
\coordinate (e10) at (-5.1,2.1) node at (e10)  {$e_{2q+2}$};
\coordinate (e11) at (-5.1,-2.1) node at (e11)  {$e_{2q+2k'-1}$};
\coordinate (e12) at (-4,-3.2) node at (e12)  {$e_{2q+2k'}$};
\coordinate (e13) at (4,-3.2) node at (e13)  {$e_{2k+1+a+1}$};
\coordinate (e14) at (5.1,-2.1) node at (e14)  {$e_{2k+1+a+2}$};
\coordinate (e15) at (5.1,2.25) node at (e15) [ right] {$e_{2k+1+a+2\ell}$};
\coordinate (e16) at (3.2,3) node at (e16) [right] {$e_{2k+1+a+2\ell+1}$};
\coordinate (e17) at (-3,1.5) node at (e17)  {$e_2$};
\coordinate (e18) at (-3,-1.5) node at (e18) {$e_{2k}$};
\coordinate (e19) at (3,-1.5) node at (e19) {$e_{2q+2k'+2}$};
\coordinate (e20) at (3,1.5) node at (e20)  {$e_{2q+2k'+2\ell'-1}$};

\fill
(v1) circle (2pt)
(v2) circle (2pt)
(v3) circle (2pt)
(v4) circle (2pt)
(v5) circle (2pt)
(v6) circle (2pt)
(v7)circle (2pt)
(v8) circle (2pt)
(v9) circle (2pt)
(v10)circle (2pt)
(v11) circle (2pt)
(v12) circle (2pt)
(v13) circle (2pt)
(v14)circle (2pt)
(v15) circle (2pt)
(v16) circle (2pt)
(v17) circle (2pt)
(v18) circle (2pt)
(v19) circle (2pt)
(v20) circle (2pt)
(v21)circle (2pt)
(v22) circle (2pt)
(v25) circle (2pt)
(v26) circle (2pt)
;

 \end{tikzpicture}
 }
 \end{center}
\caption{Case 3.2}
\label{case3.2}
\end{figure}
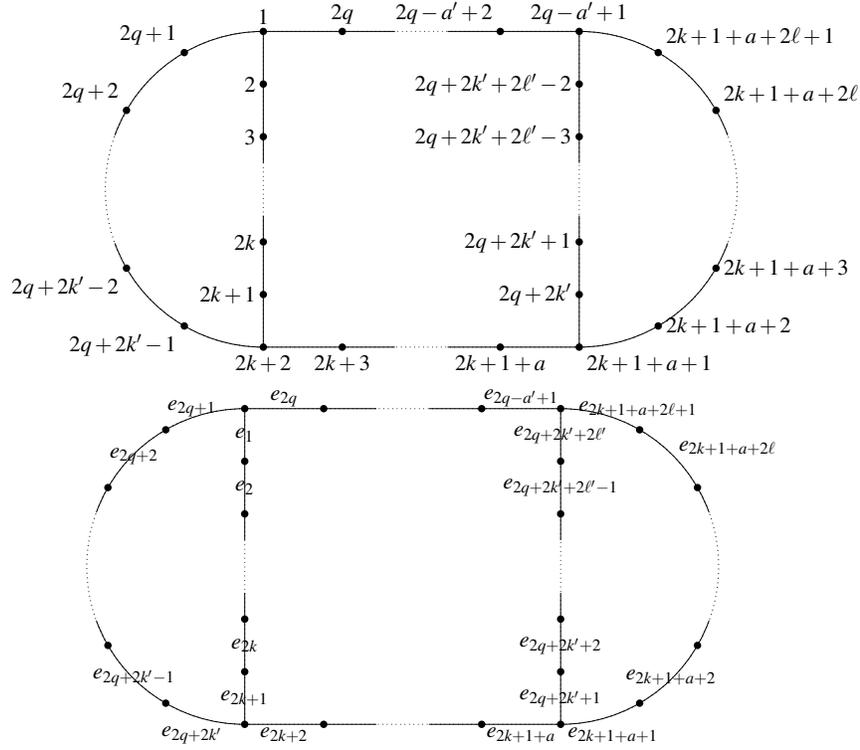
Then one has $\dim({\mathcal{P}_{G_0}})=2q+2k'+2\ell'-2-1=2q+2k+2\ell-1$.

\bigskip

\noindent
{\bf Case 3.2.1.} ($H=H_{12}$, $a=2s+1$, and $a'=2t+1$.)

Since
	\begin{eqnarray*}
	&	&\sum_{i=1}^{2k'} \dfrac{1}{2} \rho(e_{{2q}+{i}}) +\sum_{i=1}^{2\ell'}  \dfrac{1}{2} \rho(e_{2q+2k'+i}) +\sum_{i=1}^{k} \left(\dfrac{1}{3} \rho(e_{{2i-1}}) + \dfrac{2}{3}\rho(e_{2i})\right)+\dfrac{1}{3} \rho(e_{{2k+1}})\\
		&+&\sum_{i=1}^{\ell} \left(\dfrac{1}{3} \rho(e_{{2k+1+a+2i-1}}) + \dfrac{2}{3}\rho(e_{2k+1+a+2i})\right)+\dfrac{1}{3} \rho(e_{{2k+1+a+2\ell+1}})\\
		&+&\sum_{i=1}^{s} \left(\dfrac{1}{6} \rho(e_{2k+1+{2i-1}}) + \dfrac{5}{6}\rho(e_{2k+1+2i})\right)+\dfrac{1}{6} \rho(e_{2k+1+2s+1})\\	
		&+&\sum_{i=1}^{t} \left(\dfrac{1}{6} \rho(e_{2k+1+a+2\ell+1+{2i-1}}) + \dfrac{5}{6}\rho(e_{2k+1+a+2\ell+1+2i})\right)+\dfrac{1}{6} \rho(e_{2q})\\
		&=& {\eb_1}+{\eb_2}+\cdots+\eb_{m'-2} \in {\rm int}((k+\ell+k'+\ell'+s+t+1) \mathcal{P}_{G_0})\cap{\mathbb{Z}^{m'-2}},
	\end{eqnarray*}
	we obtain  ${\rm codeg}(\mathcal{P}_{G_0}) \leq q+k+\ell$.
	Hence it follows that
	\[
	{\rm deg}(\mathcal{P}_{G_0})
\geq 2q+2k+2\ell-1+1-q-k-\ell=q+k+\ell \geq q.
	\]
	
\noindent
{\bf Case 3.2.2.} ($H=H_{12}$, $a=2s$, and $a'=2t$.)

Since
	\begin{eqnarray*}
	&	&\sum_{i=1}^{2k'}  \dfrac{1}{2} \rho(e_{{2q}+{i}}) +\sum_{i=1}^{2\ell'}  \dfrac{1}{2} \rho(e_{2q+2k'+i}) +\sum_{i=1}^{k} \left(\dfrac{1}{3} \rho(e_{{2i-1}})+ \dfrac{2}{3}\rho(e_{2i})\right)\\
		&+& \dfrac{1}{3} \rho(e_{{2k+1}})+\sum_{i=1}^{\ell} \left(\dfrac{2}{3} \rho(e_{{2k+1+a+2i-1}}) + \dfrac{1}{3}\rho(e_{2k+1+a+2i})\right)+\dfrac{2}{3} \rho(e_{{2k+1+a+2\ell+1}})\\
		&+&\sum_{i=1}^{s} \left(\dfrac{1}{6} \rho(e_{2k+1+{2i-1}}) + \dfrac{5}{6}\rho(e_{2k+1+2i})\right)\\	
		&+&\sum_{i=1}^{t} \left(\dfrac{5}{6} \rho(e_{2k+1+a+2\ell+1+{2i-1}}) + \dfrac{1}{6}\rho(e_{2k+1+a+2\ell+1+2i})\right)\\
	&=&	\eb_{2k+1+a+1} + \eb_{2q-a'+1}+{\eb_1}+{\eb_2}+\cdots+\eb_{m'-2} \\
& \in& {\rm int}((k+\ell+k'+\ell'+s+t+1) \mathcal{P}_{G_0}) \cap{\mathbb{Z}^{m'-2}},
	\end{eqnarray*}
	we obtain  ${\rm codeg}(\mathcal{P}_{G_0}) \leq q+k+\ell+1$.
	Hence it follows that
	\[
	{\rm deg}(\mathcal{P}_{G_0})
\geq 2q+2k+2\ell-1+1-q-k-\ell-1=q+k+\ell -1\geq q.
	\]


\noindent
{\bf Case 3.3.} ($H=H_{13}, H_{14}$.)

Since the arguments in Case 3.1 and Case 3.2.2 remain valid even if $a'=0$,
we have $ \deg(\mathcal{P}_G)\geq q $ when $H=H_{13}$.
Furthermore, 
since the argument in Case 3.2.2 remains valid even if $a=a'=0$, we have $ \deg(\mathcal{P}_G)\geq q $ when $H=H_{14}$.

\bigskip

\noindent
{\bf Case 3.4.} ($H=H_{15}$.)

Let $a$, $b$, $c$, $d$, $e$, and $f$ be the lengths of the paths of
 the graph $G_0$ in Fig.~\ref{case3.4_1}.
Since the length of every even cycle in $G_0$ is $2q$,
we have $a+b+e+f=b+c+d+e=c+a+d+f=2q$.
This gives $a+f=b+e=c+d=q$.

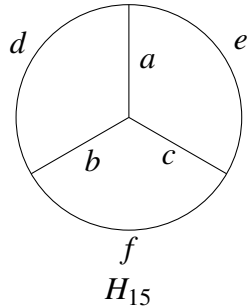
\begin{figure}[h]
\begin{center}
\begin{tikzpicture}{0.9}
\draw (0,0) circle (1.5);
\draw (90:1.5)--(0,0);
\draw (210:1.5)--(0,0);
\draw (330:1.5)--(0,0);
\coordinate (c7) at (150:2) node at (c7) [right]{$d$};
\coordinate (c8) at (30:2) node at (c8)[left] {$e$};
\coordinate (c9) at (270:2.1) node at (c9) [above] {$f$};
\coordinate (c10) at (90:0.75) node at (c10) [right] {$a$};
\coordinate (c11) at (230:0.75) node at (c11)  {$b$};
\coordinate (c12) at (315:0.75) node at (c12)  {$c$};
\coordinate (v4) at (0,-2) node at (v4) [below] {$H_{15}$};
\end{tikzpicture}
\end{center}
\caption{Case 3.4 (the length of each path)}
\label{case3.4_1}
\end{figure}

 Given positive integers $a$, $b$, and $c$ with $a,b,c < q$, 
let $G_0$ be a graph on $[3q-2]$ with the edge set
 $E(G_0)=\{e_1,\ldots, e_{3q} \}$, where 
	\[
	e_i=\begin{cases}
		\{0,v_{{s}_1}\} & i=1,\\
		\{v_{{s}_{i-1}}, v_{{s}_{i}}\} & 2 \leq i \leq q-1 \mbox{ and } i \neq a+1,\\
		\{v_{{t}_{b}},v_{s_{a+1}}\} &i=a+1,\\
		\{v_{s_{q-1}},v_{{k_{c}}}\} & i=q,\\
		\{0,v_{{t}_1}\} & i=q+1,\\
		\{v_{{t}_{i-1}}, v_{{t}_{i}}\} & q+2 \leq i  \leq 2q-1  \mbox{ and } i \neq q+b+1\\
		\{v_{{k}_{c}},v_{t_{b+1}}\} &i=q+b+1,\\
		\{v_{t_{q-1}},v_{{s_{a}}}\} & i=2q,\\
		\{0,v_{{u}_1}\} & i=2q+1,\\
		\{v_{{u}_{i-1}}, v_{{u}_{i}}\} & 2q+2 \leq i \leq 3q-1
 \mbox{ and } i \neq 2q+c+1\\
		\{v_{{s}_{a}},v_{u_{c+1}}\} &i=2q+c+1,\\
		\{v_{u_{q-1}},v_{{t_{b}}}\} & i=3q.\\
	\end{cases}
	\]
See Fig.~\ref{case3.4}.
\begin{figure}[h]
 \begin{center}
\scalebox{0.65}{
 \begin{tikzpicture}
 \draw[dotted] (0,0) circle (5);
 \draw  (45:5) arc (45:135:5);
 \draw  (165:5) arc (165:255:5);
 \draw  (285:5) arc (285:375:5);
 \draw [dotted](90:5)--(0,0);
 \draw [dotted](210:5)--(0,0);
 \draw [dotted](330:5)--(0,0);
 \draw (90:5)--(90:3);
 \draw (90:2)--(90:0);
 \draw (210:5)--(210:3);
 \draw (210:2)--(210:0);
 \draw (330:5)--(330:3);
 \draw (330:2)--(330:0);
 
 \coordinate (v1) at (0,0) node at (v1) [above right] {$0$};
 
 \coordinate (v2) at (90:1.5) node at (v2) [above right] {$v_{{s}_{1}}$};  
\coordinate (v3) at (90:3.5) node at (v3) [above right] {$v_{{s}_{a-1}}$};
 \coordinate (v4) at (90:5) node at (v4) [above right] {$v_{{s}_{a}}$};
 
\coordinate (v5) at (210:1.5) node at (v5) [above ] {$v_{{t}_{1}}$};  
\coordinate (v6) at (210:3.5) node at (v6) [above ] {$v_{{t}_{b-1}}$};
\coordinate (v7) at (210:5) node at (v7) [below left] {$v_{{t}_{b}}$};
 
\coordinate (v8) at (330:1.5) node at (v8) [above ] {$v_{{u}_{1}}$};  
\coordinate (v9) at (330:3.5) node at (v9) [above ] {$v_{{u}_{c-1}}$};
\coordinate (v10) at (330:5) node at (v10) [below right] {$v_{{k}_{c}}$};

\coordinate (v11) at (230:5) node at (v11) [below left] {$v_{{s}_{a+1}}$};
\coordinate (v12) at (250:5) node at (v12) [below left] {$v_{{s}_{a+2}}$};
\coordinate (v13) at (290:5) node at (v13) [below right] {$v_{{s}_{q-2}}$};
\coordinate (v14) at (310:5) node at (v14) [below right] {$v_{{s}_{q-1}}$};

\coordinate (v15) at (350:5) node at (v15) [below right] {$v_{{t}_{b+1}}$};
\coordinate (v16) at (370:5) node at (v16) [below right] {$v_{{t}_{b+2}}$};
\coordinate (v17) at (410:5) node at (v17) [above right] {$v_{{t}_{q-2}}$};
\coordinate (v18) at (430:5) node at (v18) [above right] {$v_{{t}_{q-1}}$};

\coordinate (v19) at (110:5) node at (v19) [above left] {$v_{{u}_{c+1}}$};
\coordinate (v20) at (130:5) node at (v20) [above left] {$v_{{u}_{c+2}}$};
\coordinate (v21) at (170:5) node at (v21) [left] {$v_{{u}_{q-2}}$};
\coordinate (v22) at (190:5) node at (v22) [left] {$v_{{u}_{q-1}}$};

\coordinate (e1) at (90:0.75) node at (e1)  {$e_1$};
\coordinate (e2) at (90:4.25) node at (e2)  {$e_a$};

\coordinate (e3) at (215:0.75) node at (e3)  {$e_{q+1}$};
\coordinate (e4) at (210:4.25) node at (e4)  {$e_{q+b}$};

\coordinate (e5) at (325:0.75) node at (e5)  {$e_{2q+1}$};
\coordinate (e6) at (330:4.25) node at (e6)  {$e_{2q+c}$};

\coordinate (e7) at (220:5) node at (e7)  {$e_{a+1}$};
\coordinate (e8) at (240:5) node at (e8)  {$e_{a+2}$};
\coordinate (e9) at (300:5) node at (e9)  {$e_{q-1}$};
\coordinate (e10) at (320:5) node at (e10) {$e_{q}$};

\coordinate (e11) at (340:5) node at (e11)  {$e_{q+b+1}$};
\coordinate (e12) at (360:5) node at (e12) {$e_{q+b+2}$};
\coordinate (e13) at (420:5) node at (e13) {$e_{2q-1}$};
\coordinate (e14) at (440:5) node at (e14) {$e_{2q}$};

\coordinate (e15) at (100:5) node at (e15) {$e_{2q+c+1}$};
\coordinate (e16) at (120:5) node at (e16) {$e_{2q+c+2}$};
\coordinate (e17) at (180:5) node at (e17) {$e_{3q-1}$};
\coordinate (e18) at (200:5) node at (e18) {$e_{3q}$};

 \fill
(v1) circle (2pt)
(v2) circle (2pt)
(v3) circle (2pt)
(v4) circle (2pt)
(v5) circle (2pt)
(v6) circle (2pt)
(v7)circle (2pt)
(v8) circle (2pt)
(v9) circle (2pt)
(v10)circle (2pt)
(v11) circle (2pt)
(v12) circle (2pt)
(v13) circle (2pt)
(v14)circle (2pt)
(v15) circle (2pt)
(v16) circle (2pt)
(v17) circle (2pt)
(v18) circle (2pt)
(v19) circle (2pt)
(v20) circle (2pt)
(v21) circle (2pt)
(v22) circle (2pt)
;
  \end{tikzpicture}
}
  \end{center}
\caption{Case 3.4}
\label{case3.4}
\end{figure}
 Then one has $\dim({\mathcal{P}_{G_0}})=3q-2-1=3q-3$.

\bigskip

\noindent
{\bf Case 3.4.1.} (all of $a$, $b$, and $c$ are odd.)

Since each $C_{4i}$ is an odd cycle, all of $q-a$, $q-b$, and $q-c$ are odd.
Let 
\[
(a,q-a,b,q-b,c,q-c)=(2a'+1,2\tilde{a}+1,2b'+1,2\tilde{b}+1,2c'+1,2\tilde{c}+1).
\]
Since
	\begin{eqnarray*}
	&	&\sum_{i=1}^{2a'+1} \dfrac{1}{2} \rho(e_{i}) +\sum_{i=1}^{2\tilde{a}+1} \dfrac{1}{2} \rho(e_{2a'+1+i}) +\sum_{i=1}^{b'} \left(\dfrac{1}{3} \rho(e_{{q+2i-1}}) + \dfrac{2}{3}\rho(e_{q+2i})\right)\\
&+&\dfrac{1}{3} \rho(e_{{q+2b'+1}})+\sum_{i=1}^{\tilde{b}} \left(\dfrac{1}{3} \rho(e_{{q+2b'+1+2i-1}}) + \dfrac{2}{3}\rho(e_{q+2b'+1+2i})\right)+\dfrac{1}{3} \rho(e_{{2q}})\\
		&+&\sum_{i=1}^{c'} \left(\dfrac{1}{6} \rho(e_{2q+{2i-1}}) + \dfrac{5}{6}\rho(e_{2q+2i})\right)+\dfrac{1}{6} \rho(e_{2q+2c'+1})\\	
		&+&\sum_{i=1}^{\tilde{c}} \left(\dfrac{1}{6} \rho(e_{2q+2c'+1+{2i-1}}) + \dfrac{5}{6}\rho(e_{2q+2c'+1+{2i}})\right)+\dfrac{1}{6}\rho(e_{3q})\\
&\in& {\rm int}((a'+\tilde{a}+b'+\tilde{b}+c'+\tilde{c}+2) \mathcal{P}_{G_0}) \cap{\mathbb{Z}^{3q-2}},
	\end{eqnarray*}
 we obtain  ${\rm codeg}(\mathcal{P}_{G_0}) \leq(3q-2)/2$.
	Hence it follows that
	\[
	{\rm deg}(\mathcal{P}_{G_0})
\geq 3q-3+1-\dfrac{3q-2}{2}=\dfrac{3q-2}{2} \geq q.
	\]

\noindent
{\bf Case 3.4.2.}  (all of $a$, $b$, and $c$ are even.)

Since each $C_{4i}$ is an odd cycle, all of $q-a$, $q-b$, and $q-c$ are odd.
Let 
\[
(a,q-a,b,q-b,c,q-c)=(2a',2\tilde{a}+1,2b',2\tilde{b}+1,2c',2\tilde{c}+1).
\]
Since
	\begin{eqnarray*}
		&&\sum_{i=1}^{2a'} \dfrac{1}{2} \rho(e_{i}) +\sum_{i=1}^{2\tilde{a}+1}  \dfrac{1}{2} \rho(e_{2a'+1+i})+\sum_{i=1}^{b'} \left(\dfrac{2}{3} \rho(e_{{q+2i-1}}) + \dfrac{1}{3}\rho(e_{q+2i})\right)\\
		&+&\sum_{i=1}^{\tilde{b}} \left(\dfrac{1}{3} \rho(e_{{q+2b'+2i-1}}) + \dfrac{2}{3}\rho(e_{q+2b'+2i})\right)
+\dfrac{1}{3} \rho(e_{{2q}})\\
&+&\sum_{i=1}^{c'} \left(\dfrac{5}{6} \rho(e_{2q+{2i-1}}) + \dfrac{1}{6}\rho(e_{2q+2i})\right)\\	
		&+&\sum_{i=1}^{\tilde{c}} \left(\dfrac{1}{6} \rho(e_{2q+2c'+{2i-1}}) + \dfrac{5}{6}\rho(e_{2q+2c'+{2i}})\right)+\dfrac{1}{6} \rho(e_{3q})\\
&\in& {\rm int}((a'+\tilde{a}+b'+\tilde{b}+c'+\tilde{c}+1) \mathcal{P}_{G_0}) \cap{\mathbb{Z}^{3q-2}},
	\end{eqnarray*}
 we obtain  ${\rm codeg}(\mathcal{P}_{G_0}) \leq (3q-1)/2$.
	Hence it follows that
	\[
	{\rm deg}(\mathcal{P}_{G_0})
\geq 3q-3+1-\dfrac{3q-1}{2}=\dfrac{3q-3}{2} \geq q.
	\]

\noindent
{\bf Case 3.4.3} (at least one of $a$, $b$, and $c$ is odd, and at least one of $a$, $b$, and $c$ is even.)

If $a$ and $b$ are odd, and $c$ is even, then
$q-a$ and $q-b$ are even since each $C_{4i}$ is an odd cycle.
By considering $c$, $q-a$, and $q-b$ instead of $a$, $b$, and $c$,
the case is reduced to Case 3.4.2.
Thus we may assume that 
$a$ is odd, and $b$ and $c$ are even.
Since each $C_{4i}$ is an odd cycle,
$q-a$ is odd and $q-b$ and $q-c$ are even.
Let 
\[(a,q-a,b,q-b,c,q-c)=(2a'+1,2\tilde{a}+1,2b',2\tilde{b},2c',2\tilde{c}).\]
Since
	\begin{eqnarray*}
	&	&\sum_{i=1}^{2a'+1} \dfrac{1}{2} \rho(e_{i})+\sum_{i=1}^{2\tilde{a}+1} \dfrac{1}{2} \rho(e_{2a'+1+i})+\sum_{i=1}^{b'} \left(\dfrac{2}{3} \rho(e_{{q+2i-1}}) + \dfrac{1}{3}\rho(e_{q+2i})\right)\\
		&+&\sum_{i=1}^{\tilde{b}} \left(\dfrac{1}{3} \rho(e_{{q+2b'+2i-1}}) + \dfrac{2}{3}\rho(e_{q+2b'+2i})\right)
+\sum_{i=1}^{c'} \left(\dfrac{5}{6} \rho(e_{2q+{2i-1}}) + \dfrac{1}{6}\rho(e_{2q+2i})\right)\\	
		&+&\sum_{i=1}^{\tilde{c}} \left(\dfrac{5}{6} \rho(e_{2q+2c'+{2i-1}}) + \dfrac{1}{6}\rho(e_{2q+2c'+{2i}})\right)\\
&\in& {\rm int}((a'+\tilde{a}+b'+\tilde{b}+c'+\tilde{c}+1) \mathcal{P}_{G_0}) \cap{\mathbb{Z}^{3q-2}},
	\end{eqnarray*}
 we obtain  ${\rm codeg}(\mathcal{P}_{G_0}) \leq 3q/2$.
	Hence it follows that
	\[
	{\rm deg}(\mathcal{P}_{G_0})
\geq \left\lceil 3q-3+1-\dfrac{3q}{2} \right\rceil = q-2 + \left\lceil \dfrac{q}{2} \right\rceil \geq q.
	\]
\end{proof}

\subsection{Toric ideals of edge rings}

In the present subsection, we study the set of generators of the toric ideal of a graph.
Let $G$ be a finite connected graph.
As explained in Section~\ref{section:ql}, 
the toric ring $K[\Pc_G]$ of $\Pc_G$ is isomorphic to the edge ring $K[G]$ of $G$.
Let $G$ be a graph with $m$ edges and let $I_G \subset K[y_1,\ldots,y_m]$ denote the toric ideal of $\Pc_G$.
The toric ideal $I_G$ has been discussed in many papers.
See \cite[Chapter 5]{binomialideals} and \cite{Vbook} for details.
In particular, the minimal set of generators of $I_G$ has been described 
using graph theoretical terminology.
 A \textit{walk} of $G$ {\em of length $q$} connecting $v_1 \in V(G)$ and $v_{q+1} \in V(G)$ is a finite sequence of the form
\[
\Gamma=(\{v_1,v_2\},\{v_2,v_3\},\ldots,\{v_q,v_{q+1}\})
\]
with each $\{v_k, v_{k+1}\} \in E(G)$.  An \textit{even walk} is a walk of even length and a \textit{closed walk} is a walk such that $v_1=v_{q+1}$.  Given an even closed walk
\[
\Gamma=(e_{i_1},e_{i_2},\ldots,e_{i_{2q}})
\]
of $G$ with each $e_k \in E(G)$,
we write $f_{\Gamma}$ for the binomial
$$f_{\Gamma}=\prod_{k=1}^{q}y_{i_{2k-1}}-\prod_{k=1}^{q}y_{i_{2k}}$$
belonging to $I_G$,
where $\pi(y_i)=\xb^{\rho(e_i)}s$.
The following lemma is due to Villarreal \cite[Proposition 3.1]{Vpaper}.

\begin{Lemma}[{\cite[Lemma~5.9]{binomialideals}}]
\label{ecw2}
The toric ideal $I_G$ of a finite connected simple graph $G$ is generated by
$\{ f_{\Gamma} : \Gamma \mbox{ is an even closed walk of }  G \}$.
\end{Lemma}

A necessary condition for $f_\Gamma$ to belong to a minimal set of binomial
generators of $I_G$ is as follows.

\begin{Proposition}[{\cite[Lemmas 5.10 and 5.11]{binomialideals} and \cite{OHT}}]
\label{ecw}
Let $\Gamma$ be an even closed walk of a graph $G$.
If the binomial $f_\Gamma$ belongs to a minimal set of binomial generators of $I_G$,
then $f_\Gamma$ is irreducible and 
$\Gamma$ visits each vertex at most twice
and satisfies one of the following{\rm :}
\begin{itemize}
\item[{\rm (i)}]
$\Gamma$ is an even cycle{\rm ;}
\item[{\rm (ii)}]
$\Gamma$ consists of two odd cycles having
exactly one common vertex{\rm ;}
\item[{\rm (iii)}]
$\Gamma$ consists of two odd cycles $C$ and $C'$ having
no common vertices and two walks that join
$v \in V(C)$ and $v' \in V(C')$.
\end{itemize}
\end{Proposition}

Note that a complete characterization for 
a minimal set of binomial generators of $I_G$ is
given in \cite{minigene}.

\begin{Lemma}
\label{2q2q2q}
Let $q  \ge 3$.
Suppose that $K[G]$ has a $q$-linear resolution.
Then the length of any even cycle in $G$ is $2q$.
\end{Lemma}

\begin{proof}
Let $C$ be an even cycle in $G$ of length $2r$.
If $2r < 2q$, then the degree of the binomial $f_C \in I_G$ is $r < q$, and hence $K[G]$ does not have a $q$-linear resolution.
Suppose that $2r > 2q$.
Then $K[G]$ does not have a $q$-linear resolution
by Lemma~\ref{degqqq}.
\end{proof}

\section{Graphs with $q$-linear resolutions}
\label{sec:last}

In the present section, we give proofs for Theorems~\ref{thm:main}, \ref{2lr}, and \ref{qlr}.
First, we give a characterization of a connected graph $G$ such that $K[G]$
is a hypersurface (i.e., ${\rm codim}(K[G])=1$).

\begin{Proposition}
\label{hypersurfaceER}
Let $G$ be a finite connected graph with $n$ vertices, $m$ edges, and non-edge blocks
$B_1,\ldots, B_s$.
Then the following conditions are equivalent{\rm :}
\begin{itemize}
\item[{\rm (i)}]
$K[G]$ is a hypersurface{\rm ;}

\smallskip

\item[{\rm (ii)}]
$\displaystyle
m=
\left\{
\begin{array}{cc}
n & \mbox{ if } G \mbox{ is bipartite,}\\
n+1 & \mbox{otherwise;}
\end{array}
\right.
$
\smallskip

\item[{\rm (iii)}]
$G$ satisfies one of the following{\rm :}
\begin{itemize}
\item[{\rm (a)}]
$s=1$ and $B_1$ is an even cycle;

\item[{\rm (b)}]
$s=1$ and $B_1$ is a non-bipartite graph
obtained by adding a path to an even cycle;

\item[{\rm (c)}]
$s=2$ and $B_1$ and $B_2$ are cycles such that 
at least one of $B_1$ and $B_2$ is an odd cycle.
\end{itemize}\end{itemize}
\end{Proposition}

\begin{proof}
Let $G$ be a connected graph.
Since ${\rm codim}(K[G]) = m-n+r(G)$,
it follows that ${\rm codim}(K[G])=1$
if and only if $m=n+1-r(G)$.
Thus we have (i) $\Leftrightarrow$ (ii).

If $G$ is bipartite, then (ii) holds if and only if $c(G)=0$
if and only if $G$ has exactly one cycle.
Hence $G$ is a bipartite graph with $n$ edges 
if and only if $G$ satisfies condition (a).
Thus  (ii) $\Leftrightarrow$ (iii) holds for bipartite graphs. 

Suppose that $G$ is not bipartite.
Then (ii) holds if and only if $G$ is bicyclic.
By Proposition~\ref{bitri} (a), 
a non-bipartite graph $G$ is bicyclic
if and only if $G$ satisfies either (b) or (c).
Thus we have (ii) $\Leftrightarrow$ (iii) holds for non-bipartite graphs. 
\end{proof}

\begin{Lemma}
\label{noeven}
Let $q  \ge 3$.
Suppose that a connected graph $G$ has no even cycles.
If $K[G]$ has a $q$-linear resolution, then
$K[G]$ is a hypersurface.
\end{Lemma}

\begin{proof}
Since $G$ has no even cycles, it is well known that each block of $G$ is either an edge or an odd cycle.
If $G$ has at most two odd cycles, then $K[G]$
is either a hypersurface or isomorphic to a polynomial ring.

Suppose that $G$ has exactly three odd cycles.
Then $G$ is tricyclic, and we may assume that $G$ is a subdivision of one of the graphs
$H_1,\ldots, H_7$ in Fig.~2.
Since the codimension of $K[G]$ is $c=2$,
the number of generators of $I_G$ is at least
 $ \binom{c+q-1}{c-1}=q+1 \ge 4$
by Lemma~\ref{EGlemma}.
Then
there exist at least four even closed walks of length $2q$ satisfying
one of the conditions (ii) and (iii) in Proposition \ref{ecw}.
Since $G$ has exactly three odd cycles,
there exists a pair of odd cycles $(C,C')$ that yields two
even closed walks $\Gamma$ of length $2q$ satisfying
condition (iii) of Proposition \ref{ecw}.
Then $G$ is a subdivision of one of the graphs
$H_1$, $H_3$, and $H_4$ in Fig.~2.
Let $\Gamma_1 = (C, W_1 , C', W_2)$ be an even closed walk of length $2q$ satisfying
 condition (iii) of Proposition \ref{ecw}.
Since $\Gamma_1$ visits each vertex at most twice, and the sum of the length of $W_1$ and $W_2$
is even, it follows that $W_1$ is a path and $W_2$ is the reverse walk of $W_1$.
Let $\Gamma_2 = (C, W_1' , C', W_2')$ be another even closed walk of length $2q$ satisfying
 condition (iii) of Proposition \ref{ecw}.
Then $W_1$ and $W_1'$ have the same length.
However, this is impossible since 
$G$ is a subdivision of one of the graphs
$H_1$, $H_3$, and $H_4$ in Fig.~2 that has no even cycles.
Thus $I_G$ has at most three generators and hence 
$K[G]$ does not have a $q$-linear resolution.

If $G$ has more than three odd cycles,
then $G$ has an induced connected subgraph $G'$ with exactly three odd cycles.
Since $K[G']$ does not have a $q$-linear resolution,
$K[G]$ does not have a $q$-linear resolution by Proposition \ref{inducedsubgraph}.
\end{proof}

We are now in a position to prove main theorems.

\begin{proof}[Proof of Theorem \ref{thm:main}]
Suppose that $K[G]$ has a $q$-linear resolution and is not a hypersurface.
By Lemmas \ref{2q2q2q} and \ref{noeven}, $G$ has at least one even cycle, and the length of any even cycle of $G$ is $2q$.
Moreover, from Lemma \ref{T2lemma}, 
there exists exactly one block $B$ of $G$ that contains an even cycle.

Suppose that $B$ is an even cycle.
Since $K[G]$ is not a hypersurface.
$G$ has at least two odd cycles $C_1$ and $C_2$.
Then there exists a tricyclic subgraph of $G$ which contains
$B \cup C_1 \cup C_2$.
By Lemma \ref{codimension2lemma}, this is a contradiction.

Suppose that $B$ is a nonbipartite graph
obtained from an even cycle by joining two vertices by a path.
If $G$ has no blocks consisting of an odd cycle, then $K[G]$ is a hypersurface.
Suppose that $G$ has a block $B'$ consisting of an odd cycle.
Then there exists a tricyclic subgraph of $G$ which contains $B \cup B'$.
By Lemma \ref{codimension2lemma}, this is a contradiction.

Thus the cyclotomic number of $B$ satisfies $c(B) \ge 3$.
Then there exists a subgraph $H$ of $B$
such that $H$ is a tricyclic graph having an even cycle.
By Lemma \ref{codimension2lemma}, this is a contradiction.
\end{proof}

\begin{proof}[Proof of Theorem~\ref{2lr}]
Let $G$ be a connected graph with $n$ vertices and $m$ edges.
In the proof of \cite[Theorem~4.6]{OHquad}, it was shown that 
$K[G]$ has a $2$-linear resolution if and only if $G$ has a subgraph 
$K_{2,\delta}$ and $I_G = I_{K_{2,\delta}}K[x_{2\delta+1},\ldots,x_m]$
by changing the indices of variables if needed.

If $G$ satisfies condition (i), then each edge $e \in E(G) \setminus E(K_{2,\delta})$
is not contained in any cycle of $G$.
Since $G$ is bipartite, $I_G$ is generated by $f_\Gamma$ where $\Gamma$ is an even cycle of $G$
by Proposition \ref{ecw}.
Hence $I_G = I_{K_{2,\delta}}K[x_{2\delta+1},\ldots,x_m]$.
Suppose that $G$ satisfies condition (ii).
Let $G'$ be a subgraph of $G$ obtained by removing $e$ from $G$
which satisfies condition (i).
Since $G'$ is a bipartite graph with $n$ vertices and $m-1$ edges,
we have $\codim (K[G']) = (m-1) -n +1=m-n = \codim (K[G]) $.
Hence we have $I_G = I_{G'}K[x_m]$ where $x_m$ corresponds to the edge $e$.
Since $G'$ satisfies condition (i), we have 
$I_G = I_{G'}K[x_m]=
 I_{K_{2,\delta}}K[x_{2\delta+1},\ldots,x_m]$.

Suppose that $K[G]$ has a $2$-linear resolution.
Then $G$ has a subgraph 
$K_{2,\delta}$ and 
$$\codim (K[G]) = \codim (K[K_{2,\delta}])=|E(K_{2,\delta})| -|V(K_{2,\delta})|+1.$$
Let $B_1, \ldots, B_s$ be the set of all blocks of $G$, where $B_1$ contains $K_{2,\delta}$.
Then 
\[
\codim (K[G]) =
r(G)-1+\sum_{i=1}^s (|E(B_i)|-|V(B_i)|+1) .
\]
Note that $|E(B_1)|-|V(B_1)|+1 \ge |E(K_{2,\delta})| -|V(K_{2,\delta})|+1$
and each $|E(B_i)|-|V(B_i)|+1$ is a nonnegative integer.

If $G$ is bipartite, then 
 $|E(B_1)|-|V(B_1)|+1 = |E(K_{2,\delta})| -|V(K_{2,\delta})|+1$
and $|E(B_i)|-|V(B_i)|+1 =0$ for any $i \ge 2$.
Hence $B_1 = K_{2,\delta}$ and $B_i$ is an edge for any $i \ge 2$.
Thus this is equivalent to satisfying condition (i).

If $G$ is not bipartite, then there exists an odd cycle $C$.
Let $e$ be an edge of $C$ which is not contained in $K_{2,\delta}$,
and let $G'$ be a subgraph of $G$ obtained by removing $e$.
Since $I_G = I_{K_{2,\delta}}K[x_{2\delta+1},\ldots,x_m]$, we have
$m-n=\codim (K[G]) = \codim (K[G'])=m-1-n+r(G')$.
Thus $G'$ should be a bipartite graph such that $K[G']$ has a $2$-linear resolution.
\end{proof}

\begin{proof}[Proof of Theorem~\ref{qlr}]
From Theorem \ref{thm:main} and Proposition~\ref{hypersurfaceER},
$K[G]$ has a $q$-linear resolution
if and only if $I_G$ is generated by a binomial of degree $q$, and $G$ satisfies
one of conditions (a), (b), and (c) of Proposition~\ref{hypersurfaceER}.
From Proposition~\ref{ecw}, we have the following:
\begin{itemize}
\item[{\rm (i)}]
If $s=1$ and $B_1$ is an even cycle, 
then $I_G$ is generated by a binomial of degree $\ell/2$,
where $\ell$ is the length of $B_1$.
Thus $K[G]$ has a $q$-linear resolution if and only if $\ell = 2q$.

\item[{\rm (ii)}]
If $s=1$ and $B_1$ is a non-bipartite graph
obtained by adding a path to an even cycle,
then  $I_G$ is generated by a binomial of degree $\ell/2$,
where $\ell$ is the length of the even cycle.
Thus $K[G]$ has a $q$-linear resolution if and only if $\ell = 2q$.

\item[{\rm (iii)}]
If $s=2$, $B_1$ is an even cycle, and $B_2$ is an odd cycle, then 
then $I_G$ is generated by a binomial of degree $\ell/2$,
where $\ell$ is the length of $B_1$.
Thus $K[G]$ has a $q$-linear resolution if and only if $\ell = 2q$.

\item[{\rm (iv)}]
Suppose that $s=2$ and $B_1$ and $B_2$ are odd cycles having one common vertex.
Let  $r_i$ be the length of $B_i$.
Then $I_G$ is generated by a binomial of degree $r_1 + r_2$.
Thus $K[G]$ has a $q$-linear resolution if and only if $r_1 + r_2 = 2q$.

\item[{\rm (v)}]
Suppose that $s=2$ and $B_1$ and $B_2$ are odd cycles
without a common vertex.
Let  $r_i$ be the length of $B_i$ and let $\ell$ be the length of the shortest path from a vertex of $B_1$
to a vertex of $B_2$.
Then $I_G$ is generated by a binomial of degree $r_1 + r_2 + 2 \ell $.
Thus $K[G]$ has a $q$-linear resolution if and only if $r_1 + r_2+ 2 \ell  = 2q$.
\end{itemize}
Hence
we have the desired conclusion.
\end{proof}

\section*{Acknowledgement}
The authors were partially supported by JSPS KAKENHI 18H01134, 19K14505, and 19J00312.

\end{document}